\documentclass[11pt]{amsart}
\usepackage{texmac}

\bbsymbols{}{A,P,G}
\bbsymbols{b}{H}
\calsymbols{}{E,L,T,B}
\calsymbols{c}{C,A,P,O,K,G,F,T}
\def\theequation{\arabic{equation}}
\def\triv{{\mathbf 1}}

\def\Ksep{K^{\text{sep}}}
\def\bigmatrix#1#2#3#4{\begin{pmatrix}#1&#2\cr#3&#4\cr\end{pmatrix}}
\begin{document}

\title{Euler factors determine local Weil representations}
\author{Tim and Vladimir Dokchitser}
\address{Dept of Mathematics, University Walk, Bristol BS8 1TW, United Kingdom}
\email{tim.dokchitser@bristol.ac.uk}
\address{Emmanuel College, Cambridge CB2 3AP, United Kingdom}
\email{v.dokchitser@dpmms.cam.ac.uk}
\date{December 20, 2011}

\begin{abstract}
We show that a Frobenius-semisimple Weil representation over a local field $K$ is determined
by its Euler factors over the extensions of $K$. The construction is explicit, and we 
illustrate it for $l$-adic representations attached to elliptic and genus 2 curves.
As an application, we construct an absolutely simple 2-dimensional abelian variety over $\Q$
all of whose quadratic twists must have positive rank, according to the
Birch--Swinnerton-Dyer conjecture.
\end{abstract}

\maketitle

\section{Introduction}
\label{sintro}

We address the question how to determine an $l$-adic representation 
over a local field $K$ from its elementary invariants. For instance, 
the local Langlands correspondence relies on the characterisation 
of representations by their local polynomials (i.e. their Euler factors) 
and `$\epsilon$-factors of pairs'. 
The main result of this paper is that they can be recovered just from their local polynomials
if one allows extensions of the ground field:

\begin{theorem}
\label{imain}
Every Frobenius-semisimple Weil representation $\rho$ over $K$ is uniquely
determined by its local polynomials 
$$
  P(\rho/F,T) = \det \bigl( 1-\Frob_{F}^{-1}T \bigm | \rho^{I_F} \bigr),
$$
where $F$ varies over finite separable extensions of $K$.
\end{theorem}

\setcounter{equation}{0}

\noindent
The proof is explicit, in the sense that it pins down a small number of extensions 
that suffice. 

Many Weil representations arise as restrictions of global representations, e.g.
representations of finite Galois groups of number fields 
(Artin representations) and $l$-adic Tate modules of elliptic curves 
and abelian varieties with potentially good reduction.
In this case, $P(\rho/K,q^{-s})$ is the local Euler factor of the corresponding 
global $L$-function. 
In the setting of elliptic curves, these
local polynomials can be computed using Tate's algorithm over extensions
$F$ of $K$ and point counting over the residue fields. The result says that this data
suffices to reconstruct $T_l E\tensor\Q_l$ explicitly as a Galois representation.

The structure of the paper is as follows.
The theorem is proved in \S\ref{smain}.
In \S\ref{selliptic} we give examples how to use the method 
to describe $l$-adic representations of elliptic curves. In \S\ref{sevil} we give
a similar example for a genus 2 curve. As an application,
we construct an absolutely simple 2-dimensional abelian variety over $\Q$
all of whose quadratic twists must have positive rank, according to the
Birch--Swinnerton-Dyer conjecture.

\begin{acknowledgements}
The first author is supported by a
Royal Society University Research Fellowship. Both authors 
would like to thank DPMMS and the University of Sydney, 
where this research was carried out.
\end{acknowledgements}

\subsection{Notation}

Throughout the paper, we use the following notation:

\medskip
\noindent
\begin{tabular}{ll}
$K$ & local field. \cr
$G_K$ & $\Gal(\Ksep/K)$, the absolute 
         Galois group of $K$. \cr
$I_K$ & the inertia subgroup of $G_K$.\cr
$\Frob_K$ & An (arithmetic) Frobenius, i.e. any element of $G_K$ acting as\cr
&  $x\mapsto x^{|k|}$ on $\bar k$, the algebraic closure of the residue field $k$ of $K$.\cr
$K^{nr}$ & the maximal unramified extension of $K$ in $\Ksep$.\cr
$l$ & prime different from the residue characteristic of $K$.\cr
$\mu_{n}, \mu_\infty$ & the set of $n^{\text{th}}$, respectively all, roots of unity.\cr
$\triv$ & trivial representation.\cr
$\rho^*$ & dual representation of $\rho$.\cr
$\tau$ & the $l$-adic tame character $\tau: I_{K}\to\Z_l$; it
maps $g$ to \smash{$
  \bigl(\frac{g (\sqrt[l^n]\pi)}{\sqrt[l^n]\pi}\bigr)_{n\ge 1}
$}\cr
&in $\invlim \mu_{l^n} \simeq \Z_l$; here $\pi$ is any uniformiser of $K$.\cr
\end{tabular}
\smallskip

Recall that the Weil group $W_K$ is the subgroup
of $G_K$ of all automorphisms that act as an integral power of
Frobenius on the residue field, see~\cite{TatN}~1.4.1.
It contains the inertia group $I_K$, and the quotient
$W_K/I_K\iso\Z$ is generated by $\Frob_K$.

Recall that a Weil representation is a representation
$\rho: W_K\to\GL_n(\C)$ such that $\rho(I_K)$ is finite.
It is called Frobenius-semisimple if the image of some (equivalently, any)
Frobenius element is diagonalisable.

The local polynomial $P(\rho,T)$ is the inverse characteristic polynomial of
$\Frob_K^{-1}$ on the inertia invariants of $\rho$,
$$
  P(\rho,T) = \det \bigl( 1-\Frob_K^{-1}T \bigm | \rho^{I_K} \bigr).
$$
Similarly, for a finite extension $F/K$, we write $P(\rho/F,T)$ for
the local polynomial of the restriction of $\rho$ to $W_F$,
$$
  P(\rho/F,T) = P(\rho|_{W_F},T).
$$

We will consider $l$-adic representations, such as $\tau$ or the $l$-adic Tate modules 
of elliptic curves and abelian varieties; they are known to be Frobenius-semisimple. 
We fix embeddings $\bar\Q_l\injects\C$ to convert
them into complex representations.

\begin{remark*}
Recall that a Weil representation $\rho$ is unramified if $\rho(I_K)=1$.
Every Frobenius-semisimple unramified representation $\rho$ is completely
determined by its local polynomial $P(\rho,T)$.
(And, conversely, any polynomial with constant term 1 comes from
such a representation.)
\end{remark*}

\section{Local factors determine Weil representations}
\label{smain}

\begin{theorem}
\label{main}
Every Frobenius-semisimple Weil representation $\rho$ is uniquely
determined by its local polynomials $P(\rho/F,T)$ over
finite separable extensions $F/K$.
\end{theorem}

\begin{proof}

{\bf Step 1: Cyclic}.
Suppose $\rho$ factors through a finite cyclic group $G=\Gal(F/K)\iso C_n$
and $F/K$ has ramification degree $e$.

By Lemma \ref{ramifieddescent}, there is a cyclic totally ramified extension $L/K$ 
of degree $e$ such that $FL/L$ is unramified of degree $n$. 
The restriction map $\Gal(FL/L)\to\Gal(F/K)$ is an isomorphism, as it is clearly 
injective and both groups have order $n$. 
So $\rho/L$ determines $\rho$, and $\rho/L$ can be recovered from its
local polynomial $P(\rho/L,T)$.

{\bf Step 2: Artin to cyclic}.
Suppose $\rho$ factors through a finite quotient, equivalently
it has local polynomial $P(\rho/F,T)=(1-T)^{\dim\rho}$ over some Galois
extension $F/K$. By character theory for the finite group $G=\Gal(F/K)$, it
is enough to recover the character of $\rho$. Thus it suffices to recover
the restriction of $\rho$ to every cyclic subgroup $\langle g\rangle \< G$
since this gives the value of the character on the conjugacy class of $g$.
This is done in Step 1.

{\bf Step 3: Twists}. Suppose $\rho$ is a twist of an Artin representations 
by a \hbox{1-dimensional} unramified character.
Over a sufficiently large Galois 
extension $F/K$, we have $P(\rho/F,T)=(1-\lambda T)^{\dim\rho}$.
Let $f$ be the residue degree of $F/K$ and define an unramified character
$\phi$ of $W_K$ by $\Frob_K\mapsto\sqrt[f]\lambda$. Then $\rho\tensor\phi^{-1}$
is an Artin representation, and it can be recovered from its local polynomials
$P(\rho\tensor\phi^{-1}/L,T)=P(\rho,T/\sqrt[f]\lambda)$ by Step 2.

{\bf Step 4: Weil to Artin}. Finally, let $\rho$ be a general Weil representation.
By Lemma \ref{weildec}, there is a unique decomposition $\rho=\bigoplus_i \rho_i$, 
such that each summand is a twist of an Artin representation $A_i$ by a 
1-dimensional unramified character $\psi_i$, and the classes of 
$\psi_i(\Frob_K)\in\C^\times/\mu_\infty$ are distinct. 
For any $F/K$,  
$$
  P(\rho/F,T) = \prod_i P(\rho_i/F,T).
$$
The roots of the $i^{\text{th}}$ factor are those 
of a given class in $\C^\times/\mu_\infty$ (namely, $[\psi_i(\Frob_K)^{f_{F/K}}]$, where 
$f_{F/K}$ is the residue degree).
By Step 3, this data determines the $\rho_i$ uniquely, and hence $\rho$ as well.
\end{proof}

\begin{lemma}
\label{ramifieddescent}
Let $F/K$ be a cyclic extension of degree $n$ and ramification \hbox{degree~$e$}.
Then there exists a cyclic totally ramified extension $L/K$ of \hbox{degree~$e$}
such that $FL/L$ is unramified of degree $n$.
\end{lemma}

\begin{proof}
It is enough prove the statement when $n$ is a prime power, since the general
case follows immediately by taking the compositum of the corresponding prime-power
extensions.

Let $\chi$ be a primitive character of $\Gal(F/K)$, and
fix a Frobenius element $\Frob_K\in G_K$.
We have two cases:
If $F/K$ is not totally ramified,
pick an unramified character $\phi$ of $G_K$ of order $n$
with $\phi(\Frob_K)=\chi(\Frob_K)^{-1}$.
Otherwise,
pick any unramified character $\phi$ of $G_K$ of order $n$.
In either case, let $L$ be the field cut out by~$\chi\phi$.

Since $\chi\phi$ is faithful on $I_{F^{nr}/K}$ and has order $e_{F/K}$,
the field $L/K$ is totally ramified of degree $e$. Moreover,
$\chi|_{G_L}=\phi^{-1}|_{G_L}$, so $FL/L$ has degree $n$, as required.
\end{proof}

\begin{lemma}
\label{weildec}
Every Frobenius-semisimple Weil representation over $K$ is a direct sum
$\rho=\bigoplus A_i\tensor \psi_i$, where each $A_i$ factors through a finite
quotient, and the $\psi_i$ are 1-dimensional and unramified.
There exists such a decomposition such that the classes
of $\psi_i(\Frob_K)$ in $\C^\times/\mu_\infty$ are distinct. The latter
decomposition is unique in the sense that the components $A_i\tensor\psi_i$
are uniquely determined up to order.
\end{lemma}

\begin{proof}
The first claim is standard: take a sufficiently large finite Galois extension $F/K$
such that $\rho/F$ is unramified. Then $\rho(\Frob_F)$ is central
in $\rho(W_K)$, so the eigenspaces of $\Frob_F$ are $W_K$-subrepresentations. Their
appropriate unramified twists give the $A_i$. 

Now, if $[\psi_i(\Frob_K)]=[\psi_j(\Frob_K)]$ in $\C^\times/\mu_\infty$, we may replace
$\psi_j$ by $\psi_i$ and $A_j$ by $A_j\tensor(\psi_j\psi_i^{-1})$ (which is also an
Artin representation) and group the two summands. Repeating this process, we get 
a decomposition as in the second claim. It is easy to see that it is unique.
\end{proof}

\begin{remark}
There are important local Galois representations that are not Weil representations, e.g. the Tate module of an elliptic curve
with multiplicative reduction. These are dealt with by using Weil-Deligne representations $W=(\rho,N)$, where $\rho$ is a Weil
representation and $N$ is a nilpotent endomorphism responsible for the `infinite' part of the inertia group \cite{TatN}. 
Their local factors are not sufficient to recover the representation: the local polynomial of $W$
is the same as that of $\ker N$ over any field, so $W$ and $\ker N$ are indistinguishable.
\end{remark}

\section{Elliptic curve examples}
\label{selliptic}

Let $E$ be an elliptic curve defined over a local field $K$ with residue
field $k=\F_q$. Consider the $l$-adic representation associated to~$E$
$$
  \rho_E: G_K \to \GL_2(\C) 
$$
defined by the Galois action on $H^1_{\text{\'et}}(E,\Q_l)\tensor\C = T_l E^*\tensor_{\Zl}\C$
for any prime $l\ne\vchar k$.

If $E/K$ has good reduction, then $\rho_E$ is unramified by the 
N\'eron-Ogg-Shafarevich criterion,
so it is completely determined by the characteristic polynomial of $\rho_E(\Frob_K)$. 
This, in turn, is determined by the number of points on the reduced 
curve $\tilde E/k$,
$$
  P(\rho_E,T) = 1 -a T + qT^2, \qquad a=q+1-|\tilde E(k)|.
$$

If $E/K$ has potentially multiplicative reduction, then it acquires split multiplicative 
reduction over some 
separable 
extension $L/K$ of degree at most 2.
The action of $G_K$ on $T_l E$ is described by the theory of the Tate curve (\cite{Sil2} Lemma V.5.2, Excs. 5.11, 5.13 or \cite{SerA} IV A.1):
$$
  \rho_E(\Frob_K)=\chi(\Frob_K)
  \bigmatrix 100{q^{-1}}, \qquad
  \rho_E(g) = \chi(g) \bigmatrix1{\tau(g)}01,  \quad g\in I_K,
$$
where $\chi:\Gal(L/K)\to\C^\times$ is the unique primitive character
and $\tau$ is the $l$-adic tame character. 

Finally, suppose $E/K$ has additive potentially good reduction. 
By Theorem \ref{main}, the $l$-adic representation $\rho_E$ can be recovered 
from the local factors $P(\rho_E/F,T)$ over extensions $F/K$. 
The proof of the theorem is constructive, and we illustrate it in this section
by determining $\rho_E$ for two specific elliptic curves. The idea is to use
several fields where $E$ acquires good reduction.

\begin{example}
\label{exell1}
Consider the elliptic curve
$$
  E/\Q_{13}: \>\> y^2 = x^3 - 26x.
$$
It has additive reduction of type \III{} and acquires good reduction over 
the field $L=\Q_{13}(\sqrt[4]{13})$, a cyclic quartic totally ramified extension of $\Q_{13}$. 
So $\rho_E: G_{\Q_{13}}\to\GL_2(\C)$ factors through
$$
  G = \Gal(\Q_{13}^{nr}(\sqrt[4]{13})/\Q_{13}) \>\>\iso\>\> \hat\Z\times C_4.
$$
This group is generated by any Frobenius element $\Phi=\Frob_L$ of $G_L$
and the inertia element $g$ that maps $\sqrt[4]{13}$ to $\iota\sqrt[4]{13}$.
(We fix $\iota$ to be the 4th root of unity congruent to 5 mod 13.)
As $G$ is abelian, Frobenius acts semisimply and $g$ acts faithfully with
determinant 1 (by the Weil pairing), $\rho_E$ is diagonalisable and
$$
  \rho_E(\Phi^{-1}) \mapsto \smallmatrix{\alpha}00{\beta}, \qquad
  \rho_E(g)=\smallmatrix{i}00{-i}.
$$
Moreover, counting points of the reduction of $E$ over the residue field $\F_{13}$ of
$L$, we find $P(\rho_E/L,T)=1+4T+13T^2$. So $\{\alpha,\beta\}=\{-2+3i,-2-3i\}$,
and it remains to determine which is which. 

As in Lemma \ref{ramifieddescent}, let $L'=\Q_{13}(\sqrt[4]{26})$, so that 
$LL'/L'$ is quartic unramified. 
Observe that $\Phi\cdot g$ is a Frobenius element in $G_{L'}$. Indeed,
it acts correctly on the residue field, and it fixes $\sqrt[4]{26}$:
\beq
  \frac{(\Phi\cdot g)(\sqrt[4]{26})}{\sqrt[4]{26}} &=&
  \frac{\Phi(g(\sqrt[4]{2})g(\sqrt[4]{13}))}{\sqrt[4]{26}} =
  \frac{\Phi(\sqrt[4]{2})}{\sqrt[4]{2}}\frac{\Phi(\iota\sqrt[4]{13})}{\sqrt[4]{13}}\cr
  &\equiv&
    (\sqrt[4]{2})^{13-1} \iota \equiv 8\cdot 5\equiv 1\mod \pi_{L'}.
\eeq
As the left-hand side is a 4th root of unity, it must be 1.

Counting points again, we find
$$
  P(\rho_E/L',T)=1-6T+13T^2=(1-(3-2i)T)(1-(3+2i)T),
$$  
so that the eigenvalues of $(\Phi\cdot g)^{-1}$ are $3\pm 2i$.
Comparing the eigenvalues of $\Phi^{-1},g$ and $(\Phi\cdot g)^{-1}$, we see that 
$\Phi^{-1}$ must act as $-2+3i$ on the $i$-eigenspace of $g$, and $\rho_E$ is given by
$$
  \rho_E(\Phi^{-1}) = \smallmatrix{-2+3i}00{-2-3i}, \qquad
  \rho_E(g) = \smallmatrix{i}00{-i}.
$$
\end{example}

\begin{example}
Consider the elliptic curve
$$
  E/\Q_2: \>\> y^2 = x^3 - x.
$$
It has additive reduction, and acquires good reduction over $F=\Q_2(E[3])$ 
(\cite{ST} Cor. 2 or \cite{Sil1} Exc. 7.9). 
This field is a splitting field of $x^8 - 288 x^4 - 6912$ (see \cite{Root2} \S3), 
equivalently of $x^8+6x^4-3$. Thus,
$$
  F = \Q_2 (i,\sqrt{3},\alpha), \qquad \alpha = \sqrt[4]{2\sqrt 3-3}.
$$
Set $z=\frac{\sqrt 3+i}2$, a fixed primitive $12^{\text{th}}$ root of unity in $F$, and 
let $\beta$ be the fourth root of $-2\sqrt 3-3$ given by 
$$
  \beta = (z^2+z)\alpha.
$$
The roots of $x^8+6x^4-3$ in $F$ are $i^k\alpha$ and $i^k\beta$ for $k=0,1,2,3$.
The Galois group $G=\Gal(F/\Q_2)$ is the semi-dihedral group of order 16, and is generated 
by an 8-cycle and an involution which fixes $\alpha$,
\beq
  s: &\alpha\mapsto -i\beta\mapsto i\alpha\mapsto\beta\mapsto 
          -\alpha\mapsto i\beta\mapsto -i\alpha\mapsto -\beta\mapsto \alpha\cr
  t: & 
  i\alpha\leftrightarrow-i\alpha, 
  \beta\leftrightarrow-i\beta,  
  -\beta\leftrightarrow i\beta.
\eeq
(In fact, $t$ is complex conjugation in $\Gal(\Q(E[3])/\Q)\iso G$.)
The inertia subgroup of $G$ is the quaternion group $Q_8$, generated by 
$s^2$ and $st$.

Fix a $\sqrt{-2}$ in $\C$. By \cite{Root2} Lemma 1, 
$$
  \rho_E \iso A \tensor \eta^{-1}, 
$$
where $\eta: G_{\Q_2}\to\C^\times$ is the unramified character with 
$\eta(\Frob_{\Q_2})=\sqrt{-2}$
and $A$ an irreducible 2-dimensional representation of $G$, which is faithful on
the inertia subgroup. Inspecting the character table of $G$,
we see that there are two possibilities for $A$, called $\rho$ and $\rho'$ in Table \ref{SD16table}.

\begin{table}[h]
$$
\begin{array}{c|rrrrrrr}
\text{order}  & 1 &  2 &  4 &  8 &  8 &  2 &  4    \cr
\text{class}  & 1 & s^4 & s^{\pm 2} &  s,s^3 & s^5,s^7 &  
                s^{2k}t & s^{2k+1}t \cr
\hline
\triv         & 1 & 1 & 1 & 1 & 1 & 1 & 1 \cr
\epsilon_1    & 1 & 1 & 1 & 1 & 1 & -1 & -1 \cr
\epsilon_2    & 1 & 1 & 1 & -1 & -1 & -1 & 1 \cr
\epsilon_3    & 1 & 1 & 1 & -1 & -1 & 1 & -1 \cr
U             & 2 & 2 & -2 & 0 & 0 & 0 & 0 \cr
\rho          & 2 & -2 & 0 & -\sqrt{-2} & \sqrt{-2} & 0 & 0 \cr
\rho'         & 2 & -2 & 0 & \sqrt{-2} & -\sqrt{-2} & 0 & 0 \cr
\end{array}
$$
\caption{Character table of $G$}
\label{SD16table}
\end{table}

As in Step 2 of the proof of Theorem \ref{main},
to determine what $A$ is we restrict to the (unique) cyclic subgroup of $G$ which 
distinguishes between $\rho$ and $\rho'$, namely 
$$
  \Gal(F/\Q_2(i)) = \langle s\rangle \iso C_8.
$$
Its inertia subgroup is $\langle s^2\rangle \iso C_4$, so $F/\Q_2(i)$ has 
residue degree 2 and ramification degree 4. 
As in Step 1 of Theorem \ref{main}, consider
$$
  L= \Q_2(i,\sqrt[4]{-3+2i}).
$$
It has the property that $FL/L$ is octic unramified. Using e.g. Artin representation 
%
machinery in Magma \cite{Magma}, we find that
$$
  P(\rho/L,T) = 1+\sqrt{-2}T-T^2, \qquad 
  P(\rho'/L,T) = 1-\sqrt{-2}T-T^2.
$$
Moreover, $E$ has good reduction over $L$, and counting points gives
$$
  P(\rho_E/L,T) = P(A\tensor\eta^{-1}/L,T) = 1-2T+2T^2.
$$
Twisting by $\eta$, we get $P(A/L,T)=1+\sqrt{-2}T-T^2$, in other words $A=\rho$ and
$$
  \rho_E \iso \rho \tensor \eta^{-1}.
$$
\end{example}

\section{A hyperelliptic curve}
\label{sevil}

We illustrate the technique for identifying $l$-adic representations in the setting of Jacobians of hyperelliptic curves. 
Their local polynomials, at least when the genus is 2, can be computed from the classification of reduction types and 
point counting on the reduced curve. 

The specific example was chosen to exhibit an absolutely simple abelian variety over $\Q$
all of whose quadratic twists have positive rank, according to the
Birch--Swinnerton-Dyer conjecture. Recall that the conjecture says that that the Mordell-Weil rank of an abelian variety $A/\Q$ 
equals to the order of vanishing of the $L$-function $L(A,s)$ at $s=1$. The $L$-function is supposed to satisfy a functional
equation relating $s\leftrightarrow 2\!-\!s$. Consequently, if the sign in the functional equation is $-1$, the order of vanishing
must be odd, and the rank must be non-zero. This sign is the global root number $w(A/\Q)$, and it is defined in terms of the
Galois action on the local $l$-adic representations $H^1_{\text{\'et}}(A/\Q_p,\Q_l)$ for each $p$.
We will construct an absolutely simple abelian variety $J/\Q$ whose global root number, 
and that of each of its quadratic twists, is $-1$.
Such examples are impossible for elliptic curves over $\Q$, but do exist for elliptic curves 
over number fields \cite{Evilquad}.
The question of existence of such abelian varieties over $\Q$ was posed to us by V. Flynn.

\begin{example}
Consider the genus 2 curve
$$
  C/\Q: y^2 + y = x^5-11x^4-6x^3+9x^2+x-1,
$$
and let $J/\Q$ be its Jacobian. 
For a prime $p$, let $\rho_J=\rho_{J,p}: G_{\Q_p}\to \GL_4(\C)$ be the Galois representation
on 
$$
  H^1_{\text{\'et}}(C/\Q_p,\Q_l)\!\tensor_{\Ql}\!\C\>\>\iso\>\>
  H^1_{\text{\'et}}(J/\Q_p,\Q_l)\!\tensor_{\Ql}\!\C\>\>\iso\>\>
  (T_l J)^*\!\tensor_{\Z_l}\!\C\qquad(l\ne p).
$$
We will show that 
\begin{enumerate}
\item[(a)]
$J$ has bad reduction only at 13 and 2633.
\item[(b)]
$J$ is absolutely simple.
\item[(c)]
At $p=2633$ the representation $\rho_{J}=\rho_{J,p}$ is, in a suitable basis,
$$
  \rho_J(\Frob_{\Q_p}^{-1})={\smaller[2]
  \begin{pmatrix}
    1&0&0&0\cr 0&1&0&0 \cr 0&0&p&0 \cr 0&0&0&p
  \end{pmatrix}}, \qquad
  \rho_J(g)={\smaller[2]\begin{pmatrix}
    1&0&\tau(g )&0\cr 0&1&0&\tau(g ) \cr 0&0&1&0 \cr 0&0&0&1
  \end{pmatrix}}
$$
for all $g\in I_{\Q_p}$, where $\tau$ is the $l$-adic tame character.
\item[(d)]
At $p=13$ the representation $\rho_{J}=\rho_{J,p}$ factors through
$$
  G = \Gal(\Q_p^{nr}(\sqrt[4]{13})/\Q_p) \>\>\iso\>\> \hat\Z\times C_4.
$$
This group is generated by any Frobenius element $\Phi$ of $G_{\Q_p(\sqrt[4]{13})}$
and the inertia element $g$ that maps $\sqrt[4]{13}$ to $\iota\sqrt[4]{13}$.
(We fix $\iota$ to be the 4th root of unity congruent to 5 mod 13.)
In a suitable basis,
$$
  \def\q#1{\qquad\llap{#1}}
  \rho_J(\Phi^{-1})={\smaller[2]
  \begin{pmatrix}
    \q{-2-3i}&0&0&0\cr 0&\q{-2+3i}&0&0 \cr 0&0&\q{-2-3i}&0 \cr 0&0&0&\q{-2+3i}
  \end{pmatrix}}, \qquad
  \rho_J(g)={\smaller[2]\begin{pmatrix}
    1&0&0&0\cr 0&1&0&0\cr 0&0&-i&0 \cr 0&0&0&i
  \end{pmatrix}}.
$$
\item[(e)]
The root number of every quadratic twist of $J/\Q$ is $-1$.
Assuming the Birch-Swinnerton-Dyer conjecture for abelian varieties, every quadratic
twist of $J$ has positive Mordell-Weil rank.
\end{enumerate}
\end{example}

\begin{proof}
(a)
Using Magma \cite{Magma} and Sage \cite{Sage}, we find that
the curve $C$ has discriminant $13^3\cdot 2633^2$
and reduction types
\In{0}-\III-0 at 13, and \In{1\text{-}1\text{-}0} at 2633 in the
Namikawa-Ueno classification \cite{NU}. 

(b) Stoll (\cite{Sto} Lemma 3) gives an explicit criterion for $J$ to be absolutely 
irreducible: it suffices to find a prime $p$ of good reduction such that the local factor
$f(T)=P(\rho_{J,p},T)$ is irreducible, and such that the 
modified resultant $\Res_T(f(T),f(Tx))/(x-1)^{2\dim J}$ 
has no irreducible monic factors in $\Z[x]$ of constant term 1. Counting points over 
$\F_{3^n}$, we find that for $p=3$,
$$
  P(\rho_{J,p},T) = 1 - T + 2T^2 - 3T^3 + 9T^4, 
$$
which is irreducible, and 
$$
  \frac{\Res_T(f(T),f(Tx))}{(x-1)^{4}} = 3^8 (x^4+\frac 43 x^3+x^2+\frac 43 x+1)^2
    (x^4+x^3+\frac{16}9 x^2+x+1),
$$
fulfulling the criterion.

(c)
Analysis at $p=2633$: Either checking by bare hands or by using the classification 
of reduction types, we see that the equation for $C$ defines a regular model at $p$.
The reduced curve is
$$
  \tilde C: \>\> y^2 = (x-2344)^2 (x-645)^2 (x-1952),
$$
so the special fibre is a $\P^1$ with two
self-intersections. The slopes at both singular points are $\F_p$-rational, 
so $\tilde C$ has $q+1-2$ points over any extension $\F_q/\F_p$. The local polynomial
is therefore
$$
  P(C/\Q_p,T) = (1-T)^2.    
$$
In particular, the invariant subspace $\rho_{J,p}^{I_p}$ is 2-dimensional, 
with trivial action of Frobenius. The action of the inertia group on 
the whole space is described in Namikawa-Ueno \cite{NU} p. 179, and the action 
of Frobenius is forced by its action on inertia invariants and the relation
$\Frob_{\Q_p} g\Frob_{\Q_p}^{-1}=g^p$ for $g$ in the tame inertia quotient. 

(d) Analysis at $p=13$. The reduced curve is
$$
  \tilde C:\>\> y^2 = (x+2)^3 (x^2+9x+6),
$$
and its normalisation is an elliptic curve
\begingroup\def\theequation{$\dagger$}
\begin{equation}\label{E1eq}
  E_1/\F_{13}:\>\> y^2 = (x+2) (x^2+9x+6).
\end{equation}
\endgroup
From the Namikawa-Ueno classification \cite{NU} p. 161,
$$
  \rho_{J,13} = \rho_{J,13}^{I_{13}} \oplus W
$$
with 2-dimensional summands, as an $I_{13}$-representation. Moreover, inertia
acts on $W$ as $\smallmatrix 01{-1}0$, that is through the cyclic tame
quotient of order 4.
Since the inertia subgroup is normal,
the action of Frobenius necessarily preserves this decomposition.
Using Magma \cite{Magma}, we find
$$
  P(\rho_{J,13},T) = 1+4 T+13T^2,
$$
which describes the action of $G_{\Q_{13}}$ on the inertia invariant subspace $\rho_{J,13}^{I_{13}}$:
$$
  \Phi^{-1}\mapsto \smallmatrix {-2-3i}00{-2+3i}, \qquad
  g\mapsto\smallmatrix 1001\>\text{ for }\>g\in I_{13}.
$$

Next, we reconstruct the Galois action on the whole space from the local polynomials over 
suitable extensions of $\Q_{13}$. 
Consider $L=\Q_{13}(\sqrt[4]{13})$ and $L'=\Q_{13}(\sqrt[4]{26})$.
Over both fields the
curve $C$ has reduction of type \In0-\In0-1, 
its Jacobian has good reduction, and the special fibre of $C$ is two elliptic curves
meeting at a point.
One of these elliptic curves over $L$ and over $L'$ is the same curve $E_1$ as above.
The other one is, respectively,
$$
  E_2:   y^2 = x^3-5x, \qquad
  E'_2:  y^2 = x^3-x.
$$
Their local polynomials are
\beq
  P(E_1/\F_{13},T)=1+4T+13T^2, \cr
  P(E_2/\F_{13},T)=1+4T+13T^2, \cr
  P(E'_2/\F_{13},T)=1-6T+13T^2.
\eeq
By counting points of $\tilde C, E_1, E_2$ and $E'_2$ over $\F_{13^n}$, 
considering the corresponding $\zeta$-functions and computing $H^0$
and $H^2$ of the special fibres, we find
\beq
  P(\rho_{J,13}/L,T) = (1+4T+13T^2)^2, \cr
  P(\rho_{J,13}/L',T) = (1+4T+13T^2)(1-6T+13T^2).
\eeq
Therefore the eigenvalues on $W$ of $g$, $\Frob_L^{-1}(=\Phi^{-1})$ and $\Frob_{L'}^{-1}$ are,
respectively,
$$
  \{i,-i\}, \qquad \{-2-3i,-2+3i\}, \qquad \{3-2i,3+2i\}.
$$
Exactly as in the end of Example \ref{exell1}, it follows that 
$\Phi^{-1}$ must act as $-2+3i$ on the $i$-eigenspace of $g$ on $W$
and as $-2-3i$ on the $(-i)$-eigenspace.
This completes the description of $\rho_{J,13}$.

(e)
Let $\chi: G_\Q\to \pm 1$ be a character of order 1 or 2, and let $J_\chi$ be the 
quadratic twist of $J$ by $\chi$. Thus $J_\chi$ is the Jacobian of
$$
  C_d: dy^2 = x^5-11x^4-6x^3+9x^2+x-\frac34,
$$
where $\Q(\sqrt d)$ is the field cut out by $\chi$. As a $G_\Q$-representation
\hbox{$T_l J_\chi\!\iso\!T_l J\!\tensor\!\chi$,} and
$$
  \rho_{J_\chi,p} \iso \rho_{J,p} \tensor \chi
$$
for every prime $p$. The global root number $w(J_\chi/\Q)$ is, by definition, the product
of local root numbers $w_v(J_\chi)= w(J_\chi/\Q_v)$ over all places of $\Q$,
$$
  w(J_\chi) = \prod_v w_v(J_\chi).   
$$
The local root number is the `sign' of the local $\epsilon$-factor,
$w_v(J_\chi) = \frac{\epsilon_v(J_\chi,\psi,\mu)}{|\epsilon_v(J_\chi,\psi,\mu)|}$,
and we refer to Tate \cite{TatN} for the definition and basic properties of 
local \hbox{$\epsilon$-factors}; see also \cite{Tamroot}, Appendix A.

We claim that, independently of $\chi$,
$$
  w_v(J_\chi) = \biggl\{\begin{tabular}{ll}
    $-1$ & if $v=13$, \cr
    $+1$ & otherwise,
  \end{tabular}
$$
so the global root number of any quadratic twist of $J$ is $-1$, as required.

For $v=\infty$, the local root number is defined in terms of the Hodge structure of $J_\chi$,
and for an abelian variety $A/\Q$ it is simply $(-1)^{\dim A}$; 
see e.g. \cite{SabR} Lemma 2.1. In our case $\dim J_\chi=2$ and $w_v(J_\chi) =+1$.

At all primes $v\ne 13$, the abelian variety $J$ is semistable,
and the root number computation is standard: see e.g.
\cite{Tamroot} Prop. 3.23 with $\tau=\chi$ and $X(\cT^*)=\triv\oplus\triv$ for $v=2263$
and 0 else; note also that $w(\chi)^4=1$ by Lemma~\ref{lemquad}~(1).

At $v=13$, the representation $\rho_{J,13}$ is described above in (d). 
Observe~that 
$$
  \rho_{J,13}\iso \rho_{E,13}\oplus \rho_{E',13},
$$
where $E$ is the curve \eqref{E1eq} lifted to $\Q_{13}$,
$$
  E: y^2 = (x+2)(x^2+9x+6),
$$
and $E'$ is the curve in Example \ref{exell1},
$$
  E': y^2 = x^3 - 26x.
$$
The first one has good reduction, and the second one has additive reduction of type \III.
In the terminology of \cite{Evilquad}, 
$E$ is `lawful good' and $E'$ is `lawful evil' (see \cite{Evilquad} Classification 3).
In other words, the local root numbers are $w(E)=+1$, $w(E')=-1$ and $w(E/F)=w(E'/F)=1$ for 
every quadratic extension $F/\Q_{13}$. Now, if $\chi=\triv$, then
$$
  w_{13}(J\tensor\chi) = w(\rho_{J,13})=w(\rho_{E,13})w(\rho_{E',13}) = w(E)w(E') = -1.
$$
If $\chi$ has order 2 and $F$ is the field cut out by $\chi$, 
then by Lemma \ref{lemquad} (2),
$$
  w_{13}(J\tensor\chi) = w(E\tensor\chi)w(E'\tensor\chi) = 
    \frac{w(E/F)}{w(E)\chi(-1)}\>\frac{w(E'/F)}{w(E')\chi(-1)}
    = -1.    
$$
This completes the proof.
\end{proof}

\begin{lemma}
\label{lemquad}
Let $K$ be a local field and $\chi\!:\!G_K\!\to\!\pm 1$ a character of order~\hbox{$\le 2$}.
\begin{enumerate}
\item
$w(\chi)^2 = \chi(-1) = \pm 1$, where 
$\chi(-1)$ denotes $\chi$ evaluated on the image of $-1$ under the local reciprocity 
map $K^\times \to G_K^{ab}$.
\item
Suppose $\chi\ne \triv$ and let $F/K$ be the quadratic extension cut out by $\chi$.
For a Weil representation $V/K$ of even dimension $2g$, 
$$
  w(V) w(V\tensor \chi) = w(V/F) \chi(-1)^g.
$$
\end{enumerate}

\begin{proof}
(1) By the determinant formula \cite{TatN} 3.4.7,
$$
  w(\chi)^2  = w(\chi\oplus\bar\chi) = \chi(-1).
$$
(2) Write $\Ind$ for the induction of representations from $G_F$ to $G_K$. 
By inductivity of root numbers in degree 0,
\beq
  w(V) w(V\tensor \chi) &=& w(\Ind V/F) 
    = w(\Ind (V\ominus\triv^{\oplus 2g}/F)) w(\Ind\triv^{\oplus 2g}) \cr
    &=& \frac{w(V/F)}{w(\triv/F)^{2g}} w(\triv)^{2g} w(\chi)^{2g} 
    = w(V/F) \chi(-1)^g.     
\eeq
\end{proof}
\end{lemma}


\end{document}